\input ./style/arxiv-vmsta.cfg
\documentclass[numbers,compress,v1.0.1]{vmsta}

\volume{3}
\issue{3}
\pubyear{2016}
\firstpage{213}
\lastpage{221}
\doi{10.15559/16-VMSTA61}

\newtheorem{thm}{Theorem}

\theoremstyle{definition}
\newtheorem{remark}{Remark}

\startlocaldefs

\allowdisplaybreaks

\endlocaldefs

\begin{document}
\begin{frontmatter}

\title{On spectra of probability measures generated by~GLS-expansions}

\author{\inits{M.}\fnm{Marina}\snm{Lupain}}\email{marinalupain@npu.edu.ua}
\address{National Pedagogical Dragomanov University, Pyrogova Str. 9, Kyiv, Ukraine}


\markboth{M. Lupain}{On spectra of probability measures generated by GLS-expansions}

\begin{abstract}
We study properties of distributions of random variables with
independent identically distributed symbols of generalized L\"uroth
series (GLS) expansions (the family of GLS-expansions contains
L\"uroth expansion and $Q_\infty$- and $G_{\infty}^2$-expansions). To
this end, we explore fractal properties of the family of Cantor-like
sets $C[\mathit{GLS},V]$ consisting of real numbers whose GLS-expansions contain
only symbols from some countable set $V \subset N \cup\{0\}$, and
derive exact formulae for the determination of the
Hausdorff--Besicovitch dimension of $C[\mathit{GLS},V]$. Based on these results,
we get general formulae for the Hausdorff--Besicovitch dimension of the
spectra of random variables with independent identically distributed
GLS-symbols for the case where all but countably many points from the
unit interval belong to the basis cylinders of GLS-expansions.
\end{abstract}

\received{5 August 2016}
\revised{27 September 2016}
%
\accepted{27 September 2016}
\publishedonline{26 October 2016}
\end{frontmatter}
\section{Introduction}
During the last 20 years, many authors studied singularly continuous
probability measures generated by different expansions of real numbers
(see, e.g., \cite{BPT,NT,P98,Torbin2002,TorbinUMJ2005,TorbinSP2007,TorbinPratsiovytyi1992}).
All these measures are the distributions of random variables of the form
\[
\xi=\varDelta^{F}_{\xi_1\xi_2\dots\xi_k\dots},
\]
where $\{\xi_k\}$ are independent or Markovian, and $F$ stands for some
expansion of real numbers. For the case of expansions over finite
alphabets, fractal properties of the spectra of the corresponding
measures are relatively well studied. For the case of infinite
alphabets, the situation is essentially more complicated. In \cite
{IP96} and \cite{NT}, it has been shown that even for self-similar
$Q_\infty$-expansion and for i.i.d.\ case, the Hausdorff--Besicovitch
dimension of the corresponding spectra cannot be calculated in a
traditional way (as a root of the corresponding equation), and formulae
for the Hausdorff dimension of the measure $\mu_\xi$ are also unknown.

In this paper, we generalize results from \cite{IP96} and \cite{NT} for
the case of distributions of random variables with independent
identically distributed GLS digits
\[
\xi=\varDelta^{\mathit{GLS}}_{\xi_1\xi_2\dots\xi_k\dots}
\]
and get general formulae for the determination of the
Hausdorff--Besicovitch dimension of spectra of $\xi$ for the case where
all but countably many points from the unit interval belong to the
basis cylinders of GLS-expansion.

\section{On GLS-expansion and fractal properties of related probability
measures}

Let $Q_{\infty}=(q_0,q_1,\dots,q_n,\dots)$ be an infinite stochastic
vector with positive coordinates.
Let us consider a countable sequence $\varDelta_i=[a_i,b_i]$ of intervals
such that $\mathit{Int}(\varDelta_i) \cap \mathit{Int}(\varDelta_j) = \emptyset\ (i\neq j)$ and
$|\varDelta_i|=q_i$. The sets $\varDelta_i$ are said to be cylinders of
GLS-expansion (generalized L\"uroth series).

Let us remark that the placement of cylinders of rank 1 is completely
determined by the preselected procedure.

For every cylinder $\varDelta_{i_1}$ of rank 1, we consider a sequence of
nonoverlapping closed intervals $\varDelta_{i_1i_2}\subset\varDelta_{i_1}$
such that
\[
\frac{|\varDelta_{i_1i_2}|}{|\varDelta_{i_1}|}=q_{i_2}
\]
and the placement of $\varDelta_{i_1i_2}$ in $\varDelta_{i_1}$ is the same as
$\varDelta_{i_1}$ in $[0;1]$.
The closed intervals $\varDelta_{i_1i_2}$ are said to be cylinders of rank
2 of the GLS-expansion.

Similarly, for every cylinder of rank $(n-1)$ $\varDelta_{i_1i_2\dots
i_{n-1}}$, we consider the sequence of nonoverlapping closed intervals
$\varDelta_{i_1i_2\dots i_n}\subset\varDelta_{i_1i_2\dots i_{n-1}}$
such that
\[
\frac{|\varDelta_{i_1i_2\dots i_n}|}{|\varDelta_{i_1i_2\dots
i_{n-1}}|}=q_{i_n}, \quad i\in N\cup\{0\},
\]
and the placement of $\varDelta_{i_1i_2\dots i_n}$ in
$\varDelta_{i_1i_2\dots i_{n-1}}$ is the same as $\varDelta_{i_1}$ in $[0;1]$.

The closed intervals $\varDelta_{i_1i_2\dots i_n}$ are said to be
cylinders of rank $n$ of the GLS-expansion.
From the construction it follows that
\[
|\varDelta_{i_1i_2...i_n}|=q_{i_1}\cdot q_{i_2}\cdot\ldots\cdot
q_{i_n}\leq(q_{\max})^n \rightarrow0 \quad (n\rightarrow
\infty),
\]
where $q_{\max}:=\max_i q_i$.

So, for any sequence of indices $\{i_k\}$ ($i_k \in N\cup\{0\})$,
there exists the sequence of embedded closed intervals
\[
\varDelta_{i_1}\subset\varDelta_{i_1i_2}\subset\varDelta_{i_1i_2i_2}
\subset\cdots \subset\varDelta_{i_1i_2\dots i_k}\subset\cdots
\]
with $|\varDelta_{i_1i_2\dots i_k}|\rightarrow0,\ k\rightarrow\infty$.
Therefore, there exists a unique point $x \in[0,1]$ that belongs to
all these cylinders.

Conversely, if $x \in[0,1]$ belongs to some cylinder of rank $k$ for
any $k\in N$ and $x$ is not an end-point for any cylinder, then there
exists a unique sequence of the cylinders
\[
\varDelta_{i_1(x)}\supset\varDelta_{i_1(x)i_2(x)}\supset\varDelta _{i_1(x)i_2(x)i_3(x)}
\supset\cdots\supset\varDelta_{i_1(x)i_2(x)\dots
i_k(x)}\supset\cdots
\]
containing $x$, and
\[
x=\bigcap_{k=1}^{\infty}\varDelta_{i_1(x)i_2(x)\dots i_k(x)}=
\varDelta _{i_1(x)i_2(x)\dots i_k(x)\dots}.
\]
The latter expression is called the GLS-expansion of $x$ (see, e.g.,
\cite{Ar,BDK,DK,L2014,LT2014} for details).

Let us remark that the L\"uroth expansion and $Q_\infty$-expansion \cite
{IP96,NT} are particular cases of the GLS-expansion. For the
case where the ratio of lengths of two embedded cylinders of successive
ranks depends on the last index and it is a power of $\varphi=\frac
{1+\sqrt{5}}{2}$, we get the $G_{\infty}^2$-expansion of $x$ \cite{F2008}.

Let $Q_\infty=(q_0,q_1,\dots,q_n,\dots)$ be a stochastic vector with
positive coordinates, and let
$x=\varDelta_{i_1(x) i_2(x) \dots i_k(x) \dots}^{\mathit{GLS}}$ be the
GLS-expansion of $x\in[0,1]$.

Let $\{\xi_k\}$ be a sequence of independent identically distributed
random variables:
\[
P(\xi_k=i):=p_i\geq0,
\]
where
\[
\sum_{i=0}^{\infty}p_i=1.
\]

Using the sequence $\{\xi_k\}$ and a given GLS-expansion, let us
consider the random variable
\[
\xi=\varDelta_{\xi_1\xi_2...\xi_k...}^{\mathit{GLS}},
\]
which is said to be the random variable with independent identically
distributed GLS-symbols.
Let $\mu_\xi$ be the corresponding probability measure.

To investigate metric, topological, and fractal properties of the
spectrum of the random variable with independent identically
distributed GLS-symbols, let us study properties of the following
family of sets. Let $V$ be a subset of $N_0:=\{0,1,2,\dots\}$, and let
\[
C[\mathit{GLS},V]=\bigl\{x: x=\varDelta^{\mathit{GLS}}_{\alpha_1(x)\dots\alpha_k(x)\dots},\alpha _k
\in V\bigr\}.
\]

If the set $V$ is finite, then $C[\mathit{GLS},V]$ is a self-similar set
satisfying the open set condition (see, e.g., \cite{Fal}). So, its
Hausdorff--Besicovitch dimension coincides with the root of the equation
\begin{equation}
\label{root} \sum\limits
_{i \in V} q_i^x = 1.
\end{equation}

If the set $V$ is countable, then the situation is essentially more
complicated. In particular, there exist stochastic vectors $Q_{\infty}$
and subsets $V$ such that equation \eqref{root} has no roots on the
unit interval.

For example, if $q_i=\frac{A}{(i+2)\ln^2(i+2)}$ and $V=N$, then the
equation $\sum_{i \in V} q_i^x = 1$ has no roots on $[0;1]$.

\begin{thm}[]\label{thm1}
If a stochastic vector $Q_{\infty}$ and a set $V\subset N_0$ are such
that the equation $\sum_{i\in V}q_i^x=1$
has a root $\alpha_0$ on $[0,1]$, then
\[
\dim_H\bigl(C[\mathit{GLS},V]\bigr)=\alpha_0.
\]
\end{thm}
\begin{proof}

First, let us show that for any $k \in N$, the set $C[\mathit{GLS},V]$ can be
covered by cylinders of rank $k$ and that the $\alpha_0$-volume of this
covering is equal to 1.

For $k=1$, the set $C[\mathit{GLS},V]$ can be covered by cylinders of rank 1.
It easy to see that the $\alpha_0$-volume is equal to 1:
\[
\sum\limits
_{i_1\in V}|\varDelta_{i_1}|^{\alpha_0}=\sum
\limits
_{i_1 \in
V}q_{i_1}^{\alpha_0}=1.
\]

Suppose that for $k=n-1$, the $\alpha_0$-volume of the covering of
$C[\mathit{GLS},V]$ by cylinders of rank $n-1$ is equal to 1. Let us show that
for $k=n$, the $\alpha_0$-volume of the covering of $C[\mathit{GLS},V]$ by
cylinders of rank $n$ will not change. We have
\begin{align*}
\sum\limits
_{i_j\in V}|\varDelta_{i_1i_2...i_{n-1}i_n}|^{\alpha_0}&=\sum
\limits
_{i_j\in V}(q_{i_1}q_{i_2}\dots
q_{n-1}q_{n})^{\alpha_0}
\\
&=\sum\limits
_{i_1\in V}q_{i_1}^{\alpha_0}\cdot\sum
\limits
_{i_j\in V}(q_{i_1}q_{i_2}\dots
q_{n-1})^{\alpha_0}=1 .
\end{align*}

So, for any $\varepsilon>0$, we get
\[
H_{\varepsilon}^{\alpha_0}\bigl(C[\mathit{GLS},V]\bigr)\leq1.
\]
Hence,
\[
H^{\alpha_0}\bigl(C[\mathit{GLS},V]\bigr)\leq1.
\]

By the definition of the Hausdorff--Besicovitch dimension we get
\[
\dim_H\bigl(C[\mathit{GLS},V]\bigr)\leq\alpha_0.
\]

Let us show that $\dim_H(C[\mathit{GLS},V])\geq\alpha_0$. To this end, let us
consider sets $V=\{i_1,\dots,i_k,\dots\}$, $V_k=\{i_1,\dots,i_k\},\
k\geq2,\ k \in N$, and the sequence $C[\mathit{GLS},V_k]$ of subsets of $C[\mathit{GLS},V]$.
For all $k\geq2,\ k\in N$, we have
\[
C[\mathit{GLS},V_k] \subset C[\mathit{GLS},V_{k+1}],
\]
and, therefore,
\[
\dim_H\bigl(C[\mathit{GLS},V_k]\bigr)\leq\dim_H
\bigl(C[\mathit{GLS},V_{k+1}]\bigr).
\]

Let $\dim_H(C[\mathit{GLS},V_k])=\alpha_k$. The sets $C[\mathit{GLS},V_k]$ are
self-similar and satisfy the open set condition (OSC). Hence, the
Hausdorff--Besicovitch dimension $\alpha_k$ of $C[\mathit{GLS},V_k]$ coincides
with the solution of the equation
\[
\sum\limits
_{i \in V_k}q_i^x=1.
\]
It is clear that $\alpha_2<\alpha_3<\cdots<\alpha_k<\cdots$ and $\alpha
_k<\alpha_0$. So, the sequence $\{\alpha_k\}$ is increasing and
bounded. Therefore, there exists a limit $\lim_{k\rightarrow
\infty}\alpha_k=\alpha^*$.

It is clear that $\alpha^*\leq\alpha_0$ because $\alpha_k < \alpha_0\
(\forall k \in N)$. Let us prove that $\alpha^*=\alpha_0$.

Assume the opposite: let $\alpha^* <\alpha_0$. Then there exists $\alpha
'$ such that $\alpha^*<\alpha' <\alpha_0$.
Then $\sum_{i \in V_k}q_i^{\alpha'}<1$ for all $k \in N$. Since
$\sum_{i\in V_k}q_i^{\alpha_k}=1$, we get $\sum_{i\in
V_k}q_i^{\alpha'}<1 $ for all $k \in N$. Let us consider the series
$\sum_{k=1}^{\infty} q_{i_k}^{\alpha'}$. It is clear that $\sum_{k=1}^{n}q_{i_k}^{\alpha'}<1$ for all $ n \in N$. So $\lim_{n \rightarrow\infty}\sum_{k=1}^{n}q_{i_k}^{\alpha'}
\leq1$. Therefore, $\sum_{i \in V}q_i^{\alpha'}\leq1$.

Since $q_i^{\alpha'}>q_i^{\alpha_0}$ for all $i \in V$ and $\sum_{i\in V}q_i^{\alpha_0}=1$, we get
$\sum_{i \in V} q_i^{\alpha'}>\sum_{i \in V}q_i^{\alpha
_0}=1$, which contradicts the already proven inequality $\sum_{i
\in V}q_i^{\alpha'}\leq1$.
This proves that $\alpha^*=\alpha_0$.

Since for any $k\geq2,\ k\in N$,
\[
\alpha_k = \dim_H\bigl(C[\mathit{GLS},V_k]\bigr)
\leq\dim_H\bigl(C[\mathit{GLS},V]\bigr),
\]
we get
\[
\alpha_0\leq\dim_H\bigl(C[\mathit{GLS},V]\bigr).
\]
Thus,
\[
\alpha_0 =\dim_H\big([\mathit{GLS},V]\big).\qedhere
\]
\end{proof}

\begin{thm}[]\label{thm2}
If the matrix $Q_{\infty}$ and the set $V=\{i_1, i_2,\dots
,i_k,i_{k+1},\dots\}$ are such that equation $\sum_{i \in
V}q_i^x=1$ has no roots on $[0,1]$,
then
\[
\dim_H\bigl(C[\mathit{GLS},V]\bigr)=\lim_{k \rightarrow\infty}
\dim_H\bigl([\mathit{GLS},V_k]\bigr),
\]
where $V_k=\{i_1,i_2,\dots,i_k\},\ k \in N,\ k\geq2$.
\end{thm}

\begin{proof}
The sets $C[\mathit{GLS},V_k]$ are self-similar and satisfy the OSC. Thus the
dimension $\alpha_k$ can be obtained as a solution of the equation $\sum_{i \in V_k}q_i^x=1$. It is easy to see that
\[
\alpha_2<\alpha_3<\cdots<\alpha_{k-1}<
\alpha_k <1.
\]

Therefore, there exists the limit
\[
\lim_{k \rightarrow\infty}\alpha_k = \alpha^*.
\]
It is clear that $\alpha^*\leq\dim_H(C[\mathit{GLS},V])$ because
$C[\mathit{GLS},V_k]\subset C[\mathit{GLS},V]$ for all $k\geq2$.
Suppose that $\alpha^*<\dim_H(C[\mathit{GLS},V])$. Then there exists a number
$\alpha'$ such that $\alpha^*<\alpha'<\dim_H(C[\mathit{GLS},V])$.
It is clear that
\[
\sum\limits
_{i \in V_2}q_i^{\alpha'}<\sum
\limits
_{i \in V_3}q_i^{\alpha
'}<\cdots<\sum
\limits
_{i \in V_k}q_i^{\alpha'}<1.
\]

So, $\sum_{i \in V}q_i^{\alpha'}\leq1$.

On the other hand, $H^{\alpha'}(C[\mathit{GLS}, V])=\infty$ by the definition of
the Hausdorff--Besicovitch dimension (because $\alpha'<\dim_H(C[\mathit{GLS},V])$).
Then the $\alpha'$-dimen\-sional Hausdorff measure of the set $C[\mathit{GLS},
V]$ with respect to the family $\varPhi$ of coverings that are generated
by the GLS-expansion of the unit segment is equal to $H^{\alpha
'}(C[\mathit{GLS},\break \varPhi])=\infty$
(where the family $\varPhi$ is a locally fine system of the coverings of
the unit segment, i.e., for any $\varepsilon>0$, there exists such a
covering of $[0,1] $ by the subsets $E_j\in\varPhi$ such that
$|E_j|<\varepsilon$ and $[0,1]=\bigcup_jE_j$).
Since the set $C[\mathit{GLS}, V]$ can be covered by cylindrical segments of the
GLS-expansion with indices from $V$, we deduce that for any $M>0$,
there exists $k(M)$ such that for all $k>k(M)$, we have the inequality
\begin{align*}
\sum_{i_q \in V, q\in\{1,\ldots,k\}}|\varDelta_{i_1i_2\dots i_k}|^{\alpha'}&>M\xch{,}{.}
\\
\sum_{i_q \in V, q\in\{1,\dots,k\}}|\varDelta_{i_1i_2\dots i_k}|^{\alpha'}&=
\sum_{i_q \in V, q\in\{1,\ldots,k-1\}}|\varDelta_{i_1i_2\dots i_{k-1}}|^{\alpha'}\cdot
\sum_{i_k \in V}\varDelta_{i_k}^{\alpha'}
\\
&<\sum_{i_q \in V, q\in\{1,\dots,k-1\}}|\varDelta_{i_1i_2\dots i_{k-1}}|^{\alpha'}
\\
&=\sum_{i_q \in V, q\in\{1,\ldots,k-2\}}|\varDelta_{i_1i_2\dots i_{k-2}}|^{\alpha'}
\cdot\sum_{i_{k-1}\in V}q_{i_{k-1}}^{\alpha'}
\\
&<\sum_{i_q \in V, q\in\{1,\dots,k-2\}}|\varDelta_{i_1i_2\dots i_{k-2}}|^{\alpha'}<
\cdots
\\
&<\sum_{i_1 \in V}|\varDelta_{i_1}|^{\alpha'}=
\sum_{i_1 \in V}q_{i_1}^{\alpha'}<1.
\end{align*}

From the obtained contradiction it follows that
\[
\lim_{k\rightarrow\infty}\alpha_k= \dim_H
\bigl(C[\mathit{GLS},V]\bigr),
\]
where $\alpha_k=\dim_H(C[\mathit{GLS},V_k])$.
\end{proof}

\begin{remark}\label{rem1}
Theorems~\ref{thm1} and~\ref{thm2} can be considered as natural generalizations of
results from \cite{IP96}.
\end{remark}

\begingroup
\abovedisplayskip=5pt
\belowdisplayskip=5pt
\begin{thm}[]\label{thm3}
The Hausdorff--Besicovitch dimension can be calculated as follows\textup{:}
\[
\dim_H\bigl(C[\mathit{GLS},V]\bigr)=\sup\biggl\{x: \sum
_{i\in V}q_i^x \geq1\biggr\}
\]
for any $Q_{\infty}$ and $V=\{i_1, i_2,\dots,i_k,i_{k+1},\dots\}$.
\end{thm}

\begin{proof}

Let $\alpha_0=\dim_H(C[\mathit{GLS},V])$. Show that $\sup\{x: \sum_{i \in
V} q_i^x \geq1\}\geq\alpha_0$.
Let us consider the function
\[
\varphi(x)=\sum_{i\in V}q_i^x
\]
and denote the set
\[
A_+=\bigl\{x: \varphi(x)\geq1\bigr\}.
\]
Let $\alpha_k $ and $\alpha_{k+1}$ be the solutions of the equations
\[
\sum_{i\in V_k}q_i^x = 1
\]
and
\[
\sum_{i\in V_{k+1}}q_i^x = 1,
\]
respectively.

Let us show that $\alpha_k<\alpha_{k+1}<\alpha_0$.
If $\alpha_0$ is the solution of $\sum_{i\in V}q_i^x = 1$, then
it is easy to see that $\alpha_k<\alpha_{k+1}<\alpha_0$.
If $\sum_{i\in V}q_i^x = 1$ has no roots on $[0,1]$, then
\[
\alpha_0=\lim_{k\rightarrow\infty}\alpha_k
\]
and $\alpha_k<\alpha_{k+1}$, so that $\alpha_k<\alpha_{k+1}<\alpha_0$.

Express the function $\varphi(x)$ as follows:
\[
\varphi(x)=\underbrace{q_{i_1}^x+\cdots+q_{i_k}^x+q_{i_{k+1}}^x}_{\geq
1}+
\sum_{j=k+1}^\infty q_{i_j}^x.
\]
It is easy to see that for all $x\in[\alpha_k, \alpha_{k+1}]$
$(x<\alpha_0,\ k\in N)$, $\varphi(x)\geq1$.
Then $A_+\supset(-\infty; \alpha_0)$ and $\sup A_+\geq\alpha_0$.

Let us show that $\sup A_+\leq\alpha_0$. Suppose the opposite. If $\sup
A_+>\alpha_0$, then there exists $x_1$ such that $x_1 \in(\alpha_0;
\sup A_+]$, $x_1 \in A_+ $, and
$\varphi(x_1)\geq1$. So
\[
\sum_{i \in V}q_i^{x_1}\geq1.
\]

Since $\alpha_k$ is a solution of $\sum_{i\in V_k}q_i^x = 1$ and
$\alpha_k<\alpha_0$, we get
\[
\sum_{i\in V_k}q_i^{\alpha_k} = 1
\]
and
\[
\sum_{i\in V_k}q_i^{\alpha_0} \leq1.
\]
It is clear that
\[
\sum_{i\in V}q_i^{\alpha_0} <1
\]
and
\[
\sum_{i\in V}q_i^{x_1} \leq1.
\]

So, from the obtained contradiction it follows that
$\sup A_+=\alpha_0$.
\end{proof}

Let $\varDelta^{\mathit{GLS}}_\infty$ be the set of those $x\in[0;1]$ that do not
belong to any cylinder of the first rank of the GLS-expansion. The set
$\varDelta^{\mathit{GLS}}_\infty$ can be empty, countable, or of continuum cardinality.

Let us recall that the nonempty and bounded set $E$ is called
$N$-self-similar if it can be represented as a union of a countably
many sets $E_j\ (\dim_H (E_i\cap E_j)< \dim_H E, i \neq j)$ such that
the set $E$ is similar to the sets $E_j$ with coefficient~$k_j$.

Since the spectrum $S_\xi$ of the distribution of a random variable $\xi
$ with independent identically distributed GLS-symbols is a
self-similar or $N$-self-similar set and
$S_\xi=\break(C[\mathit{GLS},V])^{\mathit{cl}}$, we can apply the above results to calculate
the Hausdorff--Besicovitch dimension of the spectrum $S_\xi$ for the
case where $\varDelta^{\mathit{GLS}}_\infty$ is an at most countable
set.\vadjust{\eject}

So, we get the following theorem, which can be considered as a
corollary of Theorems~\ref{thm1} and~\ref{thm2}.

\begin{thm}[]\label{thm4}
Let $V:=\{i:p_i>0\}$. If $\varDelta^{\mathit{GLS}}_\infty$ is at most countable,
then the Hausdorff--Besicovitch dimension of the spectrum of the
distribution of a~random variable $\xi$ with independent identically
distributed $\mathit{GLS}$-symbols can be calculated in the following way.

1) If the equation $\sum_{i\in V}q^x_i=1$ has one root $\alpha
_0$ on $[0,1]$, then
\[
\dim_H S_\xi=\alpha_0.
\]

2) If the equation $\sum_{i\in V}q^x_i=1$ has no roots on
$[0,1]$, then
\[
\dim_H S_\xi=\lim_{k\rightarrow\infty}
\alpha_k,
\]
where $\alpha_k$ are the roots of the equations
$\sum_{i\in V_k}q^x_i=1,\ V_k=\{i_1, i_2,\dots,i_k\}\subset V$.
\end{thm}
\endgroup

\section*{Acknowledgments}
The author expresses her sincere gratitude to the referee for useful
comments and remarks.

\bibliographystyle{vmsta-mathphys}

\begin{thebibliography}{15}

\bibitem{Ar}
%
\begin{barticle}
\bauthor{\bsnm{Arroyo}, \binits{A.}}:
\batitle{Generalized \textsc{L}\"uroth expansions and a family of
\textsc{M}inkowski's question-mark functions}.
\bjtitle{Math. Acad. Sci. Paris}
\bvolume{353},
\bfpage{943}--\blpage{946}
(\byear{2015}).
\bid{doi={10.1016/\\j.crma.2015.08.008}, mr={3411226}}
\end{barticle}
%
%
\OrigBibText
%
\begin{barticle}
\bauthor{\bsnm{Arroyo}, \binits{A.}}:
\batitle{Generalized \textsc{L}\"uroth expansions and a family of
\textsc{M}inkowski's question-mark functions}.
\bjtitle{Math. Acad. Sci. Paris}
\bvolume{353},
\bfpage{943}--\blpage{946}
(\byear{2015})
\end{barticle}
%
\endOrigBibText
\bptok{structpyb}%
\endbibitem

\bibitem{BPT}
%
\begin{bbook}
\bauthor{\bsnm{Baranovskyi}, \binits{O.}},
\bauthor{\bsnm{Pratsiovytyi}, \binits{M.}},
\bauthor{\bsnm{Torbin}, \binits{G.}}:
\bbtitle{Ostrogradsky--Sierpinski--Pierce Series and Their Applications}.
\bpublisher{Naukova Dumka},
\blocation{Kyiv}
(\byear{2013})
\end{bbook}
%
%
\OrigBibText
%
\begin{bbook}
\bauthor{\bsnm{Baranovskyi}, \binits{O.}},
\bauthor{\bsnm{Pratsiovytyi}, \binits{M.}},
\bauthor{\bsnm{Torbin}, \binits{G.}}:
\bbtitle{Ostrogradsky--Sierpinski--Pierce Series and Their Applications}.
\bpublisher{Naukova Dumka},
\blocation{Kyiv}
(\byear{2013})
\end{bbook}
%
\endOrigBibText
\bptok{structpyb}%
\endbibitem

\bibitem{BDK}
%
\begin{barticle}
\bauthor{\bsnm{Barrionuevo}, \binits{J.}},
\bauthor{\bsnm{Burton}, \binits{R.M.}},
\bauthor{\bsnm{Dajani}, \binits{K.}},
\bauthor{\bsnm{Kraaikamp}, \binits{C.}}:
\batitle{Ergodic properties of generalized \textsc{L}\"uroth series}.
\bjtitle{Acta Arith.}
\bvolume{74},
\bfpage{311}--\blpage{327}
(\byear{1996}).
\bid{mr={1378226}}
\end{barticle}
%
%
\OrigBibText
%
\begin{barticle}
\bauthor{\bsnm{Barrionuevo}, \binits{J.}},
\bauthor{\bsnm{Burton}, \binits{R.M.}},
\bauthor{\bsnm{Dajani}, \binits{K.}},
\bauthor{\bsnm{Kraaikamp}, \binits{C.}}:
\batitle{Ergodic properties of generalized \textsc{L}\"uroth series}.
\bjtitle{Acta Arithmetica}
\bvolume{74},
\bfpage{311}--\blpage{327}
(\byear{1996})
\end{barticle}
%
\endOrigBibText
\bptok{structpyb}%
\endbibitem

\bibitem{DK}
%
\begin{bbook}
\bauthor{\bsnm{Dajani}, \binits{K.}},
\bauthor{\bsnm{Kraaikamp}, \binits{C.}}:
\bbtitle{Ergodic Theory of Numbers}.
\bpublisher{The Mathematic Association of America},
\blocation{Washington}
(\byear{2002}).
\bid{mr={1917322}}
\end{bbook}
%
%
\OrigBibText
%
\begin{bbook}
\bauthor{\bsnm{Dajani}, \binits{K.}},
\bauthor{\bsnm{Kraaikamp}, \binits{C.}}:
\bbtitle{Ergodic Theory of Numbers}.
\bpublisher{The Mathematic Association of America},
\blocation{Washington}
(\byear{2002})
\end{bbook}
%
\endOrigBibText
\bptok{structpyb}%
\endbibitem

\bibitem{Fal}
%
\begin{bbook}
\bauthor{\bsnm{Falconer}, \binits{K.}}:
\bbtitle{Fractal Geometry: Mathematical Foundations and Applications}.
\bpublisher{John Wiley and Sons},
\blocation{New York}
(\byear{1990}).
\bid{mr={1102677}}
\end{bbook}
%
%
\OrigBibText
%
\begin{bbook}
\bauthor{\bsnm{Falconer}, \binits{K.}}:
\bbtitle{Fractal Geometry: Mathematical Foundations and Applications}.
\bpublisher{John Wiley and Sons},
\blocation{New York}
(\byear{1990})
\end{bbook}
%
\endOrigBibText
\bptok{structpyb}%
\endbibitem

\bibitem{L2014}
%
\begin{barticle}
\bauthor{\bsnm{Lupain}, \binits{M.}}:
\batitle{Fractal properties of random variables with independent
\textsc{GLS}-symbols}.
\bjtitle{Trans. Dragomanov Pedagog. Univ., Ser. 1, Phys.-Math. Sci.}
\bvolume{16}(\bissue{1}),
\bfpage{279}--\blpage{295}
(\byear{2014})
\end{barticle}
%
%
\OrigBibText
%
\begin{barticle}
\bauthor{\bsnm{Lupain}, \binits{M.}}:
\batitle{Fractal properties of random variables with independent
\textsc{GLS}-symbols}.
\bjtitle{Transactions of Dragomanov Pedagogical University. Series 1:
Phys.-Math. Sciences}
\bvolume{16(1)},
\bfpage{279}--\blpage{295}
(\byear{2014})
\end{barticle}
%
\endOrigBibText
\bptok{structpyb}%
\endbibitem

\bibitem{LT2014}
%
\begin{botherref}
\oauthor{\bsnm{Lupain}, \binits{M.}},
\oauthor{\bsnm{Torbin}, \binits{G.}}:
On new fractal phenomena related to distributions of random variables with
independent \textsc{GLS}-symbols.
Trans. Dragomanov Pedagog. Univ., Ser. 1, Phys.-Math.
Sci.
\textbf{16}(2)
(2014)
\end{botherref}
%
%
\OrigBibText
%
\begin{botherref}
\oauthor{\bsnm{Lupain}, \binits{M.}},
\oauthor{\bsnm{Torbin}, \binits{G.}}:
On new fractal phenomena related to distributions of random variables with
independent \textsc{GLS}-symbols.
Transaction of Dragomanov Pedagogical University. Series 1: Phys.-Math.
Sciences
\textbf{16(2)}
(2014)
\end{botherref}
%
\endOrigBibText
\bptok{structpyb}%
\endbibitem

\bibitem{IP96}
%
\begin{barticle}
\bauthor{\bsnm{Nikiforov}, \binits{R.}},
\bauthor{\bsnm{Torbin}, \binits{G.}}:
\batitle{On the \textsc{H}ausdorf-f-\textsc{B}esicovitch dimension of
generalized self-similar sets generated by infinite \textsc{IFS}}.
\bjtitle{Trans. Dragomanov Pedagog. Univ., Ser. 1, Phys.-Math. Sci.}
\bvolume{13}(\bissue{1}),
\bfpage{151}--\blpage{163}
(\byear{2012})
\end{barticle}
%
%
\OrigBibText
%
\begin{barticle}
\bauthor{\bsnm{Nikiforov}, \binits{R.}},
\bauthor{\bsnm{Torbin}, \binits{G.}}:
\batitle{On the \textsc{H}ausdorf-f-\textsc{B}esicovitch dimension of
generalized self-similar sets generated by infinite \textsc{IFS}}.
\bjtitle{Transaction of Dragomanov Pedagogical University. Series 1:
Phys.-Math. Sciences}
\bvolume{13(1)},
\bfpage{151}--\blpage{163}
(\byear{2012})
\end{barticle}
%
\endOrigBibText
\bptok{structpyb}%
\endbibitem

\bibitem{NT}
%
\begin{barticle}
\bauthor{\bsnm{Nikiforov}, \binits{R.}},
\bauthor{\bsnm{Torbin}, \binits{G.}}:
\batitle{Fractal properties of random variables with independent
$\textsc{Q}_\infty$-symbols}.
\bjtitle{Theory Probab. Math. Stat.}
\bvolume{86},
\bfpage{169}--\blpage{182}
(\byear{2013}).
\bid{doi={10.1090/S0094-9000-2013-00896-5}, mr={2986457}}
\end{barticle}
%
%
\OrigBibText
%
\begin{barticle}
\bauthor{\bsnm{Nikiforov}, \binits{R.}},
\bauthor{\bsnm{Torbin}, \binits{G.}}:
\batitle{Fractal properties of random variables with independent
$\textsc{Q}_\infty$-symbols}.
\bjtitle{Theory of Probability and Mathematical Statistics}
\bvolume{86},
\bfpage{169}--\blpage{182}
(\byear{2013})
\end{barticle}
%
\endOrigBibText
\bptok{structpyb}%
\endbibitem

\bibitem{P98}
%
\begin{bbook}
\bauthor{\bsnm{Pratsiovytyi}, \binits{M.}}:
\bbtitle{Fractal Approach in Investigations of Singular Distributions}.
\bpublisher{Dragomanov Pedagogical University},
\blocation{Kyiv}
(\byear{1998})
\end{bbook}
%
%
\OrigBibText
%
\begin{bbook}
\bauthor{\bsnm{Pratsiovytyi}, \binits{M.}}:
\bbtitle{Fractal Approach in Investigations of Singular Distributions}.
\bpublisher{Dragomanov Pedagogical University},
\blocation{Kyiv}
(\byear{1998})
\end{bbook}
%
\endOrigBibText
\bptok{structpyb}%
\endbibitem

\bibitem{F2008}
%
\begin{barticle}
\bauthor{\bsnm{Pratsiovytyi}, \binits{M.}},
\bauthor{\bsnm{Feshchenko}, \binits{O.}}:
\batitle{Topological, metric and fractal properties of probability
distributions on the set of incomplete sums of positive series}.
\bjtitle{Theory Stoch. Process.}
\bvolume{13},
\bfpage{205}--\blpage{224}
(\byear{2007}).
\bid{mr={2343824}}
\end{barticle}
%
%
\OrigBibText
%
\begin{barticle}
\bauthor{\bsnm{Pratsiovytyi}, \binits{M.}},
\bauthor{\bsnm{Feshchenko}, \binits{O.}}:
\batitle{Topological, metric and fractal properties of probability
distributions on the set of incomplete sums of positive series}.
\bjtitle{Theory of Stochastic Processes}
\bvolume{13},
\bfpage{205}--\blpage{224}
(\byear{2007})
\end{barticle}
%
\endOrigBibText
\bptok{structpyb}%
\endbibitem

\bibitem{Torbin2002}
%
\begin{barticle}
\bauthor{\bsnm{Torbin}, \binits{G.}}:
\batitle{Fractal properties of the distributions of random variables with
independent \textsc{Q}-symbols}.
\bjtitle{Trans. Dragomanov Pedagog. Univ., Ser. 1, Phys.-Math. Sci.}
\bvolume{3},
\bfpage{241}--\blpage{252}
(\byear{2002})
\end{barticle}
%
%
\OrigBibText
%
\begin{barticle}
\bauthor{\bsnm{Torbin}, \binits{G.}}:
\batitle{Fractal properties of the distributions of random variables with
independent \textsc{Q}-symbols}.
\bjtitle{Transaction of Dragomanov Pedagogical University. Series 1:
Phys.-Math. Sciences}
\bvolume{3},
\bfpage{241}--\blpage{252}
(\byear{2002})
\end{barticle}
%
\endOrigBibText
\bptok{structpyb}%
\endbibitem

\bibitem{TorbinUMJ2005}
%
\begin{barticle}
\bauthor{\bsnm{Torbin}, \binits{G.}}:
\batitle{Multifractal analysis of singularly continuous probability measures}.
\bjtitle{Ukr. Math. J.}
\bvolume{57},
\bfpage{837}--\blpage{857}
(\byear{2005}).
\bid{doi={10.1007/s11253-005-0233-4}, mr={2209816}}
\end{barticle}
%
%
\OrigBibText
%
\begin{barticle}
\bauthor{\bsnm{Torbin}, \binits{G.}}:
\batitle{Multifractal analysis of singularly continuous probability measures}.
\bjtitle{Ukrainian Math. J.}
\bvolume{57},
\bfpage{837}--\blpage{857}
(\byear{2005})
\end{barticle}
%
\endOrigBibText
\bptok{structpyb}%
\endbibitem

\bibitem{TorbinSP2007}
%
\begin{barticle}
\bauthor{\bsnm{Torbin}, \binits{G.}}:
\batitle{Probability distributions with independent \textsc{Q}-symbols and
transformations preserving the Hausdorff dimension}.
\bjtitle{Theory Stoch. Process.}
\bvolume{13},
\bfpage{281}--\blpage{293}
(\byear{2007}).
\bid{mr={2343830}}
\end{barticle}
%
%
\OrigBibText
%
\begin{barticle}
\bauthor{\bsnm{Torbin}, \binits{G.}}:
\batitle{Probability distributions with independent \textsc{Q}-symbols and
transformations preserving the Hausdorff dimension}.
\bjtitle{Theory of Stochastic Processes}
\bvolume{13},
\bfpage{281}--\blpage{293}
(\byear{2007})
\end{barticle}
%
\endOrigBibText
\bptok{structpyb}%
\endbibitem

\bibitem{TorbinPratsiovytyi1992}
%
\begin{botherref}
\oauthor{\bsnm{Torbin}, \binits{G.}},
\oauthor{\bsnm{Pratsiovytyi}, \binits{M.}}:
Random variables with independent $\textsc{Q}^*$-symbols. In:
Random Evolutions: Theoretical and Applied
Problems, pp. 95--104, Institute for Mathematics of NASU
(1992)
\end{botherref}
%
%
\OrigBibText
%
\begin{botherref}
\oauthor{\bsnm{Torbin}, \binits{G.}},
\oauthor{\bsnm{Pratsiovytyi}, \binits{M.}}:
Random variables with independent $\textsc{Q}^*$-symbols.
Institute for Mathematics of NASU: Random Evolutions: Theoretical and Applied
Problems,
95--104
(1992)
\end{botherref}
%
\endOrigBibText
\bptok{structpyb}%
\endbibitem

\end{thebibliography}
%

\end{document}